\documentclass[a4paper,10pt,reqno]{amsart}

\usepackage[utf8]{inputenc}
\usepackage[T1]{fontenc}
\usepackage{amsthm}
\usepackage{amsmath}
\usepackage{amssymb}
\usepackage[inline]{enumitem}
\usepackage{comment}
\usepackage{hyperref}
\usepackage{fancyhdr}
\usepackage{mathrsfs}
\usepackage{stmaryrd}
\usepackage[normalem]{ulem}
\usepackage{xcolor}

\theoremstyle{definition}

\newtheorem{theorem}{Theorem}[section]
\newtheorem{lemma}[theorem]{Lemma}
\newtheorem{proposition}[theorem]{Proposition}
\newtheorem{corollary}[theorem]{Corollary}

\theoremstyle{definition}
\newtheorem{definition}[theorem]{Definition}

\newtheorem{remark}[theorem]{Remark}

\definecolor{blue-url}{RGB}{0,0,100}
\definecolor{red-url}{RGB}{100,0,0}
\definecolor{green-url}{RGB}{0,100,0}
\definecolor{light-yellow}{RGB}{255,255,128}
\definecolor{light-blue}{RGB}{193,255,255}
\definecolor{light-red}{RGB}{239,83,80}

\hypersetup{
	pdftitle={Some applications of a new approach to factorization},
	pdfauthor={L.~Cossu},
	pdfmenubar=false,
	pdffitwindow=true,
	pdfstartview=FitH,
	colorlinks=true,
	linkcolor=blue-url,
	citecolor=green-url,
	urlcolor=red-url
}

\renewcommand{\emptyset}{\varnothing}
\renewcommand{\setminus}{\smallsetminus}
\renewcommand{\,}{\kern 0.1em}

\DeclareMathOperator{\fix}{\textup{\textsc{f}ix}}
\DeclareMathOperator{\rfix}{\textup{r.\textsc{f}ix}}
\DeclareMathOperator{\End}{\textsc{e}nd}
\DeclareMathOperator{\id}{{\rm id}}
\DeclareMathOperator{\rk}{rk}
\DeclareMathOperator{\hgt}{ht}

\providecommand\llb{\llbracket}
\providecommand\rrb{\rrbracket}

\providecommand{\R}{\mathbb{R}}
\newcommand{\fin}{\mathrm{fin}}

\newcommand{\evid}[1]{\textsf{#1}}

{\newline\vspace{\abovedisplayskip}\hbox to \textwidth\bgroup\hss$\displaystyle}
{$\hss\egroup\vspace{\belowdisplayskip}}

\makeatletter
\DeclareFontFamily{OMX}{MnSymbolE}{}
\DeclareSymbolFont{MnLargeSymbols}{OMX}{MnSymbolE}{m}{n}
\SetSymbolFont{MnLargeSymbols}{bold}{OMX}{MnSymbolE}{b}{n}
\DeclareFontShape{OMX}{MnSymbolE}{m}{n}{
	<-6>  MnSymbolE5
	<6-7>  MnSymbolE6
	<7-8>  MnSymbolE7
	<8-9>  MnSymbolE8
	<9-10> MnSymbolE9
	<10-12> MnSymbolE10
	<12->   MnSymbolE12
}{}
\DeclareFontShape{OMX}{MnSymbolE}{b}{n}{
	<-6>  MnSymbolE-Bold5
	<6-7>  MnSymbolE-Bold6
	<7-8>  MnSymbolE-Bold7
	<8-9>  MnSymbolE-Bold8
	<9-10> MnSymbolE-Bold9
	<10-12> MnSymbolE-Bold10
	<12->   MnSymbolE-Bold12
}{}

\let\llangle\@undefined
\let\rrangle\@undefined
\DeclareMathDelimiter{\llangle}{\mathopen}%
{MnLargeSymbols}{'164}{MnLargeSymbols}{'164}
\DeclareMathDelimiter{\rrangle}{\mathclose}%
{MnLargeSymbols}{'171}{MnLargeSymbols}{'171}
\makeatother

\begin{document}
\title{Some applications of a new approach to factorization}
\author{Laura Cossu}
\address{Department of Mathematics and Scientific Computing, University of Graz | Heinrichstrasse 36/III, 8010 Graz, Austria}
\email{laura.cossu@uni-graz.at}
\urladdr{https://sites.google.com/view/laura-cossu/}

%
%
%
%
%

\begin{abstract}
\noindent{As highlighted in a series of recent papers by Tringali and the author, fundamental aspects of the classical theory of factorization can be significantly generalized by blending the languages of monoids and preorders. Specifically, the definition of a suitable preorder on a monoid allows for the exploration of decompositions of its elements into (more or less) arbitrary factors.

We provide an overview of the principal existence theorems in this new theoretical framework. Furthermore, we showcase additional applications beyond classical factorization, emphasizing its generality. In particular, we recover and refine a classical result by Howie on idempotent factorizations in the full transformation monoid of a finite set.}
\end{abstract}
\maketitle
\thispagestyle{empty}

\section{Introduction}\label{sec:intro}
The classical theory of factorization investigates the arithmetic structure of commutative domains and their associated monoids of invertible ideals and modules. It is well known that if a commutative domain $D$ satisfies the ascending chain condition (ACC) on principal ideals (ACCP), then every non-zero non-invertible element $a\in D$ can be written as a non-empty finite product of the form $a=u_1 \cdots u_n$, where the $u_i$'s are {\it irreducible} elements of $D$, called {\it atoms} (that is, non-units that do not factor as a product of two non-units). Then $n$ is called a factorization length and the set $\mathsf L (a) \subseteq \mathbb{N}^+$ of all possible factorization lengths of $a$ is called the length set of $a$. If $D$ is Mori (i.e, $D$ satisfies the ACC on divisorial ideals), then all length sets are finite and the study of their structure is a key topic in factorization theory. A domain is Krull if and only if it is a completely integrally closed Mori domain, if and only if the multiplicative monoid of its non-zero elements is a Krull monoid. In Krull domains, every non-invertible element admits a factorization into atoms, and these factorizations are finitely many modulo units. A domain is factorial (i.e., every non-invertible element admits an essentially unique factorization into atoms) if and only if it is Krull with trivial class group. In general, it is the class group of a Krull domain and the distribution of height-one prime ideals in the classes that control the arithmetic of the monoid and the structure of its length sets. It is interesting to add that, in large classes of modules, direct-sum decompositions into indecomposable summands can be understood and studied successfully as a factorization problem into atoms. For instance, the monoid of finitely generated projective right $R$-modules over a semilocal ring $R$, where the operation is induced by the direct sum, is a Krull monoid. This allows the study of its arithmetic through the lens of factorization theory (see \cite{Fa06a}, \cite{Fa12a}, and \cite{BaWi13} for additional details). Since factorization problems in commutative domains are essentially multiplicative in nature, the technical machinery required for their investigation has been developed in the context of commutative and cancellative monoids, as outlined in the monograph \cite{Ge-HK06a}. Recall that a (not necessarily commutative) monoid $H$ is cancellative if $xz\ne zy$ and $zx\ne zy$ for all $x, z, y \in H$ with $x \ne y$.

In the last two decades, classical factorization theory, which originated in algebraic number theory and whose roots trace back to the 1960s, has broadened its scope in two main directions. Firstly, consider the semigroup $\mathcal I (D)$ representing non-zero ideals of a commutative domain $D$ with standard ideal multiplication. The semigroup $\mathcal I (D)$ is cancellative and factorial if and only if $D$ is a Dedekind domain, and it is cancellative if and only if $D$ is an almost Dedekind domain. However, for a general Noetherian domain, $\mathcal I (D)$ is not cancellative, yet it satisfies the property that, for $I$ and $J$ non-zero ideals of $D$,  $I J = I$ implies $J=D$. Similarly, numerous semigroups of modules lack cancellativity but satisfy the condition that $M \oplus N \cong M$ implies $N$ is the zero module. To address the study of such objects, commutative unit-cancellative monoids were introduced in \cite{F-G-K-T17}, and significant portions of classical factorization theory were developed within this broader framework (for a recent survey in this context, refer to \cite{GeZh19}). Note that the underlying idea of commutative unit-cancellative monoid comes from the concept of strongly reduced monoid in \cite{RGU99}. The second direction involves non-commutative ring theory. Questions revolve around determining whether all elements in a non-commutative ring have finitely many factorizations or whether all their length sets are finite. See, for example, \cite{Ba-Sm15} and \cite{Be-He-Le17}, together with \cite{B-B-N-S23a} and \cite{Sm19a}, where the focus is on Noetherian prime rings.

Factorization theory has witnessed further developments in the last few years. The new explorations have been motivated by a simple observation. In various areas of mathematics, it is often meaningful to decompose certain objects (of a diverse nature) into distinguished ``building blocks''. These building blocks, in a certain sense, cannot be further decomposed, but they are not necessarily ``irreducible'' in the sense of the classical theory. Here is a (non-exhaustive) list of examples:
\begin{enumerate}[label=\textup{(\alph{*})}, mode=unboxed]
\item\label{1} It is a basic result in algebra (see, e.g., \cite[Proposition (19.20)]{Lam01}) that, for a possibly non-commutative ring $R$, every Artinian or Noetherian $R$-module can be expressed as a direct sum of finitely many indecomposable submodules. Here, as usual, an $R$-module $M$ is considered \evid{indecomposable} if $M$ is neither a zero $R$-module nor the direct sum of two non-zero $R$-modules.

\vskip 0.05cm

\item\label{2} By the Lasker–Noether theorem (see, e.g., \cite[Chapter 7]{Hal-Ko98}), every proper ideal of a commutative Noetherian ring $R$ can be decomposed as an intersection, called primary decomposition, of finitely many primary ideals. Recall that an ideal $I$ of $R$ is \evid{primary} if it is a proper ideal and for each pair of elements $x,y\in R$ such that $xy\in R$, either $x$ or some power of $y$ lies in $I$.

\vskip 0.05cm

\item\label{3} By Erdos' seminal paper \cite{Er68}, every singular matrix over a field factorizes into a finite product of idempotent matrices (here and later, a matrix is always a square matrix and it is said to be singular if its determinant is zero). Subsequent work has extended this result to singular matrices over broader classes of domains (see, e.g., \cite{Laf83}, \cite{Fo91}, and \cite{Co-Za22}).

\vskip 0.05cm

\item\label{4} Howie established in \cite[Theorem I]{Ho66} that every non-invertible map on a finite set $X$ is a composition of idempotent maps fixing all but one element of $X$.

\vskip 0.05cm

\item\label{5} Every non-trivial permutation of a set of finite cardinality $n$ is a composition of less than $n$ transpositions (see, e.g., \cite[Proposition 2-35]{Rot06}).

\vskip 0.05cm

\item\label{6} The classical result commonly known as the Cartan–Dieudonné theorem (see \cite{Cartan} and \cite{Dieudonne}) establishes that every orthogonal transformation in an $n$-dimensional symmetric bilinear space can be described as the composition of at most $n$ reflections.

\vskip 0.05cm

\item\label{7} The decomposition of a numerical semigroup as an intersection of irreducible numerical semigroups is studied in \cite{RoBr02}. Irreducible numerical semigroups are either symmetric or pseudo-symmetric and have attracted the attention of many mathematicians.

\vskip 0.05cm

\item\label{8} It is proved in \cite[Proposition 2]{BGM21} that monotone functions over a finite poset $P$ decompose as a sum of characteristic functions of irreducible upper sets of $P$. 

\end{enumerate}
A method to effectively address factorization problems, stemming from both the traditional contexts outlined in the first two paragraphs and the settings \ref{1}-\ref{8}, within the same theoretical framework, is provided by the combined use of the language of monoids and preorders, as firstly done by Tringali in \cite{Tr20(c)}. This is part of a rather extensive program, already initiated in embryonic form in \cite{Fa-Tr18}, \cite{Tr19}, and \cite{An-Tr18}, aimed at extending the classical theory of factorization beyond its current boundaries. The new general approach to factorization based on preorders is the outcome of the collaborative work of Tringali and the author, who laid its foundations in a series of recent papers. In particular, the articles \cite{Tr20(c)}, \cite{Co-Tr-21(a)}, and \cite{Tri23(a)} focus on existence results and their applications, while \cite{Cos-Tri-2023(a)} and \cite{Cos-Tri-2023(b)} concentrate on the study of algebraic and arithmetic properties related to the non-uniqueness of factorization. Most notably, a concept of ``minimal factorization'', along with corresponding arithmetic invariants, was introduced in the latter two papers to counteract the blow-up phenomena typical of the arithmetic of non-cancellative monoids. To date, the terminology and abstract concepts of the new theory have already found application in the study of Boolean lattices \cite{Aj-Go23}, of the ideal class monoid of a numerical semigroup \cite{Casab-Danna-GarSan-2023}, and of special monoids of sets called power monoids \cite{Cos-Tri-2023(c)}.

In the present paper, we go back to handling results concerning the existence of specific factorizations. In Section \ref{sec: survey}, we recall the definition of a premonoid along with the associated fundamental concepts of units, quarks, and irreducibles. We then provide an overview of the main existence theorems (Theorems \ref{thm: FTF} and \ref{thm: strongly Artinian}) for factorizations into irreducibles on premonoids. This is done by extending the chain condition mentioned at the beginning to the new theoretical framework. In Section \ref{sec: examples}, we discuss several known applications of these theorems, highlighting the unifying role of the premonoid paradigm in the study of factorization. These applications include in fact classical factorization results as well as factorization results involving non-atomic (even idempotent) building blocks in non-cancellative monoids. In Section \ref{sec: idempotent}, after reviewing the definition of the $\rfix$-preorder from \cite{Co-Tr-21(a)}, we examine the full transformation monoid of a finite set. In Theorem~\ref{thm: a refinement of Howie}, we establish, as a corollary of Theorem~\ref{thm: strongly Artinian}, a quantitative refinement of the classical result by Howie mentioned in \ref{4}. The same preorder, when restricted to the monoid of permutations of a finite set, allows us to recover as a consequence of the same abstract theorem the factorization theorem alluded to in \ref{5}. Lastly, in Section \ref{sec: Cartan Dieudonné}, we consider the $\rfix$-preorder on the endomorphism ring of a free module of finite rank over a principal ideal domain. In particular, we show that the celebrated Cartan-Dieudonné theorem (see \ref{6} and Corollary~\ref{cor: CD}) can be obtained as a specialization of Theorem~\ref{thm: strongly Artinian}. These additional applications of the existence theorems for factorizations in premonoids, involving idempotent and invertible factors in non-commutative and non-cancellative monoids, complement the previous articles on the same topic and further support the robustness of the new theory.

\section{Premonoids and abstract factorization theorems}\label{sec: survey}
We take a monoid as a semigroup with an identity element. Unless otherwise specified, monoids are assumed to be multiplicative and do not possess any specific properties such as commutativity or cancellativity. We refer to Howie's monograph \cite{Ho95} for basics on semigroup theory. 

Let $H$ be a monoid with \evid{identity} $1_H$. An element $x\in H$ is a \evid{unit} of $H$ if $xy=yx=1_H$ for some $y\in H$, and it is a \evid{non-unit} otherwise. We use $H^\times$ to denote the group of units of $H$, and we say that $H$ is \evid{Dedekind-finite} if $xy\in H^\times$, for some $x,y\in H$, implies $x,y\in H^\times$. 
An element $e\in H$, on the other hand, is \evid{idempotent} if $e^2=e$. The only idempotent element of $H$ that is also a unit is the identity $1_H$. A non-unit element $a$ of $H$ is called an \evid{atom} if $a\ne xy$ for all non-units $x,y\in H$; and an \evid{irreducible} if $a\ne xy$ for all non-units $x,y\in H$ with $HxH\ne HaH\ne HyH$. Notably, every atom is an irreducible, and conversely, every irreducible in a cancellative commutative monoid is an atom (see Remark~\ref{rem:premonoids}\ref{rem:premonoids(1)} for further details). This explains the interchangeable use of the terms ``atom'' and ``irreducible'' in the context (commutative and cancellative) of classical factorization theory. A principal right ideal of a monoid $H$ is a set of the form $aH$, where $a$ is an element of $H$. Similarly, sets of the form $aH$ and $HaH$ are called principal left and principal (two-sided) ideals of $H$, respectively. The monoid $H$ is then said to satisfy the ascending chain condition on principal right ideals (ACCPR) if there does not exist an infinite sequence of principal right ideals of $H$ that is strictly increasing with respect to inclusion. Analogous definitions apply to the ascending chain condition on principal left ideals (ACCPL) and the ascending chain condition on principal ideals (ACCP), where principal right ideals are replaced by principal left ideals and principal ideals, respectively. We will see in the following how these chain conditions, which coincide in the commutative setting, play a pivotal role in the study of factorization.
Through the paper, $\mathbb N$ is the set of non-negative integers, and for $a,b\in \mathbb N\cup\{\infty\}$ we let $\llb a,b\rrb:=\{n\in\mathbb{N}: a\le n\le b\}$ be the \evid{discrete interval} from $a$ to $b$. Clearly $\llb a, b\rrb=\emptyset$ whenever $a$ is strictly smaller than $b$. 

In this section, we illustrate how the synergistic use of the language of monoids and preorders leads to general existence theorems for factorizations, with  applications far beyond the scope of classical theory. 

\subsection{Premonoids, quarks, and irreducibles.} A \evid{preorder} on a set $X$ is a reflexive and transitive binary relation $\preceq$ on $X$; for two elements $x,y \in X$, we write $x \prec y$ to mean that $x \preceq y$ and $y\npreceq x$, while we say that $x$ and $y$ are \evid{$\preceq$-equivalent} if $x\preceq y\preceq x$. A preorder $\preceq$ on $X$ is said to be \evid{Artinian} if there is no sequence $(x_k)_{k\ge 0}$ of elements of $X$ such that $x_{k+1} \prec x_{k}$ for every $k\in \mathbb{N}$. 

Following \cite[Definition 3.4]{Tr20(c)}, we call \evid{premonoid} a pair $\mathcal H := (H, \preceq)$, where $H$ is a monoid (the \evid{ground monoid} of $\mathcal{H}$) and $\preceq$ is a preorder on its underlying set. Note that we do not require any compatibility between the preorder $\preceq$ and the multiplication in $H$, and we say that $\mathcal{H}$ is an \evid{Artinian premonoid} if $\preceq$ is Artinian on $X$. The binary relation $\mid_H$ defined on a monoid $H$ by $x \mid_H y$ if and only if $x\in H$ and $y \in HxH$ is a preorder we refer to as the \evid{divisibility preorder} on $H$. Accordingly, we call \evid{divisibility premonoid} of $H$ the premonoid $H_{\rm div}:=(H,\mid_H)$. The divisibility preorder, and consequently divisibility premonoids, play a central role in the study of factorization. In particular, a divisibility premonoid $H_{\rm div}$ is Artinian if and only if it satisfies the ACCP, as defined above (see \cite[Remark 3.11(4)]{Tr20(c)}).

Given a premonoid $\mathcal H := (H, \preceq)$, an el\-e\-ment $u \in H$ is a \evid{$\preceq$-unit} of $H$ (or a \evid{unit} of $\mathcal H$) if it is $\preceq$-equivalent to $1_H$; otherwise, $u$ is a \evid{$\preceq$-non-unit} (or a \evid{non-unit} of $\mathcal H$). We write $\mathcal H^\times$ for the set of units of $\mathcal H$. A $\preceq$-non-unit $a \in \allowbreak H$ is called a \evid{$\preceq$-quark} of $H$ (or a \evid{quark} of $\mathcal H$) if there is no $\preceq$-non-unit $b$ with $b \prec a$; and is a \evid{$\preceq$-ir\-re\-duc\-i\-ble} \evid{of degree $s$} of $H$ (or an \evid{ir\-re\-duc\-i\-ble} \evid{of degree $s$} of $\mathcal H$) for a certain integer $s \ge 2$, if $a \ne x_1 \cdots x_k$ for every $k \in \llb 2, s \rrb$ and for all $\preceq$-non-units $x_1, \ldots, x_k \in H$ with $x_1 \prec a, \ldots, x_k \prec a$. We will simply refer to a $\preceq$-ir\-re\-duc\-i\-ble of degree $2$ as a \evid{$\preceq$-ir\-re\-duc\-i\-ble} of $H$ (or \evid{irreducible} of $\mathcal H$). In the course of the paper, with a slight abuse of notation, we will sometimes refer to the elements of a premonoid to denote the elements of its ground monoid.
The above no\-tions were introduced in \cite[Definition 3.6]{Tr20(c)} and \cite[Definition 3.1]{Co-Tr-21(a)}; in the next remark, we summarize some facts related to them.

\begin{remark}\label{rem:premonoids}

\begin{enumerate*}[label=\textup{(\arabic{*})}, mode=unboxed]
\item\label{rem:premonoids(0)} In the context of premonoids, it is possible to define additional classes of objects beyond those mentioned earlier, which are of limited interest for this note but that are important in other settings (see \cite[Definition 3.6]{Tr20(c)}). It is still worth highlighting that, in a premonoid $\mathcal{H}=(H,\preceq)$, a $\preceq$-non-unit $a\in H$ is called a \evid{$\preceq$-atom} of $H$ (or an \evid{atom} of $\mathcal{H}$) if $a\ne xy$ for all $\preceq$-non-units $x,y\in H$. Understanding the interrelation among irreducibles, atoms, and quarks of a premonoid $\mathcal H$ is often crucial for a deeper comprehension of various phenomena. It is clear from the definitions, for example, that both atoms and quarks are irreducibles. However, an irreducible may not necessarily be either an atom or a quark, and similarly, a quark may not be an atom. Importantly, this observation holds true even in the case of divisibility premonoids; refer to \cite[Remark 3.7(4), Proposition 4.11(iii), and Theorem 4.12]{Tr20(c)} for further details. Note, in addition, that every irreducible of degree $s$ in a premonoid $\mathcal{H}$ is an irreducible of degree $k$ for every $k\in\llb 2,s\rrb$, and that every quark is an irreducible of degree $s$ for every $s\ge 2$. Theorem~\ref{thm: strongly Artinian} gives sufficient conditions for the reverse implication of the latter statement to hold. 
\end{enumerate*}

\vskip 0.05cm

\begin{enumerate*}[label=\textup{(\arabic{*})}, mode=unboxed, resume]
\item\label{rem:premonoids(1)}  As mentioned in the introduction, classical factorization theory primarily studies factorizations into atoms in commutative and cancellative monoids. In more general terms, however, it revolves around investigating the divisibility premonoid of a \evid{Dedekind-finite} monoid $H$. A key observation in this regard is that, under the hypothesis of Dedekind-finiteness, a $\mid_H$-unit is a unit of $H$. Thus, in this case, a $\mid_H$-atom is an (ordinary) atom, i.e., a non-unit of $H$ that cannot be written as a product of two non-units, and a $\mid_H$-irreducible is an (ordinary) irreducible in the sense recalled at the beginning of the section (cf.~\cite[Remark 3.7(2)]{Tr20(c)}). More generally, if the monoid $H$ is \evid{acyclic}, i.e., $uxv \ne x$ for all $u, v, x \in H$ such that $u$ or $v$ is a non-unit, then we get from \cite[Corollary 4.4]{Tr20(c)} that $\mid_H$-ir\-re\-duc\-i\-bles, $\mid_H$-atoms, $\mid_H$-quarks, and atoms coincide.
Note that if $H$ is acyclic or cancellative, it is also \evid{unit-cancellative} as defined in \cite[Sect.~2.1, p.~256]{Fa-Tr18}. i.e., $xy \ne x \ne yx$ for all $x,y \in H$ with $y \notin H^\times$. However, the reverse may not hold true, as illustrated in Example 5.6 of \cite{Cos-Tri-2023(a)}. As mentioned in the introduction, much of the existing literature on factorization in non-cancellative monoids has primarily focused on the unit-cancellative scenario. In the commutative setting, acyclicity and unit-cancellativity are equivalent, but this is not true in the non-commutative case, as demonstrated in \cite[Example 4.8]{Tr20(c)}. We observe, to conclude, that since every unit-cancellative monoid is Dedekind-finite, so is every acyclic monoid.
\end{enumerate*}

\vskip 0.05cm

\begin{enumerate*}[label=\textup{(\arabic{*})}, mode=unboxed, resume]
\item\label{rem:premonoids(1bis)}There are examples of premonoids devoid of atoms. In the classical framework of (divisibility premonoids of) commutative and cancellative monoids, those with no atoms are referred to as \evid{antimatter}. For instance, the additive monoid of non-negative rationals is antimatter and contains no irreducibles. Consider instead the divisibility premonoid of the commutative monoid $(\mathcal{P}(X), \cup)$, where $\mathcal{P}(X)$ is the power set of the non-empty set $X$. Since every element of $\mathcal{P}(X)$ is idempotent, there are no atoms. On the other hand, every singleton in $\mathcal{P}(X)$ is irreducible.
\end{enumerate*}

\vskip 0.05cm

\begin{enumerate*}[label=\textup{(\arabic{*})}, mode=unboxed,resume]
\item\label{rem:premonoids(2)} Given a premonoid $\mathcal{H}=(H,\preceq)$, a pair $\mathcal K := (K, \preceq_K)$ is said to be a \evid{subpremonoid of $\mathcal{H}$} if $K$ is a submonoid of $H$ and $\preceq_K$ is the \evid{restriction} of $\preceq$ to $K$ (i.e., the preorder on $K$ defined by taking $x \preceq_K y$ if and only if $x, y \in K$ and $x \preceq y$). It follows from the definitions that $\mathcal{K}^\times = K \cap \mathcal{H}^\times$. In fact, since $K$ is a submonoid of $H$, we have $1_K = 1_H$ and hence $u\in K$ is a $\preceq_K$-unit if and only if $1_H\preceq u\preceq 1_H$. Note, in passing, that the restriction of the divisibility preorder $\mid_H$ on a monoid $H$ to a submonoid $K$ of $H$ does not generally coincide with the divisibility preorder on $K$. This occurs, for instance, when $K$ is a divisor-closed submonoid of $H$.
\end{enumerate*}
\end{remark}

\subsection{Existence factorization theorems} We say that a premonoid $\mathcal{H}=(H,\preceq)$ is \evid{$s$-factorable} for some integer $s\ge 2$ (respectively, \evid{factorable}) if every $\preceq$-non-unit of $H$ can be written as a (non-empty, finite) product of $\preceq$-irreducibles of degree $s$ (respectively, of $\preceq$-irreducibles). In light of Remark~\ref{rem:premonoids}\ref{rem:premonoids(0)}, every $s$-factorable premonoid is $k$-factorable for every $k\in \llb 2,s\rrb$ so, in particular, it is factorable. The relevance of the above definitions is tied to the next theorem, first proved in \cite[Theorem 3.10]{Tr20(c)} and subsequently refined in \cite[Theorem 2.7]{Co-Tr-21(a)}. This result provides sufficient conditions for a premonoid to be factorable, and it can be applied in diverse scenarios to establish the {\it existence} of particular factorizations. Its easy proof, that we do not include in this note, is based on a basic property of Artinian preorders: if $\preceq$ is an Artinian preorder on a set $X$, then every non-empty subset $S\subseteq X$ admits a \evid{$\preceq$-minimal} element, i.e., an element $x\in S$ such that $y\nprec x$ for every $y\in S$ \cite[Remark 3.9(3)]{Tr20(c)}.

\begin{theorem}\label{thm: FTF}
Every Artinian premonoid is $s$-factorable for every integer $s\ge 2$ (and hence factorable).
\end{theorem}

By imposing stricter conditions on the premonoid $\mathcal{H}=(H,\preceq)$, it is possible to refine Theorem~\ref{thm: FTF} in order to obtain an estimate of the minimum length of a factorization into $\preceq$-irreducibles of a fixed $\preceq$-non-unit $x\in H$, i.e., of the minimum number of $\preceq$-irreducibles of $H$ whose product is $x$. In order to do so, we need to introduce two new notions.

Given a premonoid $\mathcal{H}=(H,\preceq)$ and an element $x \in H$, we denote by $\hgt_\mathcal{H}(x)$ the supremum of the set of all integers $n \ge 1$ for which there are $\preceq$-non-units $x_1, \ldots, x_n \in H$ with $x_1 = x$ and $x_{k+1} \prec x_k$ for each $k \in \llb 1, n-1 \rrb$, where we set $\sup \emptyset := 0$. We call $\hgt_\mathcal{H}(x)$ the \evid{height} of $x$ (relative to $\mathcal{H}$) and write $\hgt(x)$ in place of $\hgt_\mathcal{H}(x)$ if no confusion can arise. We say that $\mathcal{H}$ is a \evid{strongly Artinian} premonoid if the height of every element of $H$ is finite.

\begin{remark}\label{rem:height}

\begin{enumerate*}[label=\textup{(\arabic{*})}, mode=unboxed]
\item\label{rem:height(0)} If $H$ is a finite monoid (i.e., $|H|<\infty$), then every premonoid $\mathcal{H}=(H,\preceq)$ is strongly Artinian: in fact, for $x,y\in H$, $x\prec y$ implies $x\ne y$.
\end{enumerate*} 

\vskip 0.05cm

\begin{enumerate*}[label=\textup{(\arabic{*})}, mode=unboxed, resume]
\item\label{rem:height(1)} It is clear from the definition that for an element $x$ in a premonoid $\mathcal{H}=(H,\preceq)$, $\hgt(x)=1$ if and only if $x$ is a quark and $\hgt(x)=0$ if and only if $x\in \mathcal{H}^\times$. Moreover, in light of Remark~\ref{rem:premonoids}\ref{rem:premonoids(2)}, if $\mathcal{K}=(K,\preceq_K)$ is a subpremonoid of $\mathcal{H}$, then the non-units of $\mathcal{K}$ are non-units of $\mathcal{H}$. It then follows that $\hgt_\mathcal{K}(y)\le \hgt_\mathcal{H}(y)$ for every element $y\in K$, and so every subpremonoid of a strongly Artinian premonoid is strongly Artinian (cf. \cite[Remark 3.3(1)]{Co-Tr-21(a)}).
\end{enumerate*} 

\vskip 0.05cm

\begin{enumerate*}[label=\textup{(\arabic{*})}, mode=unboxed,resume]
\item\label{rem:height(2)} Every strongly Artinian premonoid is Artinian. If $\mathcal{H}=(H,\preceq)$ is strongly Artinian, then the map $\lambda: H\to \mathbb N: x\mapsto \hgt(x)$ is well defined. Moreover, $\lambda(x)<\lambda(y)$ for every $x,y\in H$ such that $x\prec y$. Thus, the preorder $\preceq$ is Artinian on $H$, otherwise there would exist an infinite sequence $(n_k)_{k\ge 0}$ of non-negative integers with $n_{k+1}<n_k$ for every $k\in \mathbb N$, which is absurd (cf. \cite[Remark 3.9(1)]{Tr20(c)}). 
\end{enumerate*}

\vskip 0.05cm

\begin{enumerate*}[label=\textup{(\arabic{*})}, mode=unboxed,resume]
\item\label{rem:height(3)} Not every Artinian premonoid is strongly Artinian. For instance, let us consider the divisibility premonoid of the Puiseux monoid $\langle 1/p \,| \,p \text{ is prime}\rangle$, i.e., the additive submonoid of $\mathbb{Q}_{\ge 0}$ generated by the reciprocals of prime numbers. This premonoid is Artinian by \cite[Theorem 4.5]{CGG21}, but it is not strongly Artinian. In fact, for every prime $p$ and every $k\in \llb 1,p-1\rrb$, $k/p$ properly divides $(k+1)/p$ with the result that $\hgt(1)\ge p-1$. Thus, $\hgt(1)=\infty$.
\end{enumerate*}
\end{remark}

We are ready to state the second main theorem, in a form slightly less general than that of \cite[Theorem 3.5]{Co-Tr-21(a)} but suitable for the purposes of this note. We include the proof for completeness.

\begin{theorem}\label{thm: strongly Artinian}
    Let $\mathcal{H}=(H, \preceq)$ be a strongly Artinian premonoid and let $s\ge 2$ be an integer such that, for every $x\in H$ that is neither a $\preceq$-unit nor a $\preceq$-quark there exist $\preceq$-non-units $y_1,\dots, y_k\in H$, with $k\in\llb 2,s\rrb$, such that $y_i\prec x$ for each $i\in\llb 1,k\rrb$, $x=y_1\cdots y_k$, and $\hgt(y_1)+\cdots+\hgt(y_k)\le \hgt(x)+k-2$. Then every $\preceq$-non-unit $y\in H$ is a product of at most $(s-1)\hgt(x)-(s-2)$ $\preceq$-quarks.
\end{theorem}
\begin{proof}
Every strongly Artinian premonoid is Artinian by Remark~\ref{rem:height}\ref{rem:height(2)}. Moreover, we get directly from the hypotheses that every $\preceq$-irreducible of degree $s$ of $H$ is a $\preceq$-quark. It then follows from Theorem~\ref{thm: FTF} that every $\preceq$-non-unit $y\in H$ is a (non-empty, finite) product of $\preceq$-quarks. It remains to prove the bound on the minimal number of $\preceq$-quarks that realize such decomposition. Let us denote by $q(y)$ the minimum length of a fac\-tor\-i\-za\-tion of a $\preceq$-non-unit $y \in H$ into $\preceq$-quarks. Note that $q(y)$ is a well-defined positive integer by the $s$-factorability of $\mathcal{H}$, and that $1 \le \hgt(y) < \infty$. Fix a $\preceq$-non-unit $x \in H$ and set $N := \hgt(x)$. We claim that $q(x) \le (s-1)N - (s-2)$. If $N = 1$, then $x$ is a $\preceq$-quark by Remark~\ref{rem:height}\ref{rem:height(1)} and the conclusion is trivial. So, let $N \ge 2$ and assume inductively that $q(y) \le (s-1)\hgt(y)-(s-2)$ for every $\preceq$-non-unit $y \in H$ with $\hgt(y)\le N-1$. Since $x$ is neither a $\preceq$-unit nor a $\preceq$-quark, there exist $k\in\llb 2,s\rrb$ $\preceq$-non-units $y_1,\dots, y_k\in H$ such that $y_i\prec x$ for each $i\in\llb 1,k\rrb$, $x=y_1\cdots y_k$, and $\hgt(y_1)+\cdots+\hgt(y_k)\le \hgt(x)+k-2$. Then $1 \le \hgt(y_i) \le N-1$ for every $i \in \llb 1,k \rrb$, and we get by the inductive hypothesis that $y_i$ factors as a product of $(s-1) \hgt(y_i) - (s-2)$ or fewer $\preceq$-quarks. Therefore
\[q(x) \leq \sum_{i=1}^k q(y_i) \le \sum_{i=1}^k \bigl((s-1)\hgt(y_i) -(s-2)\bigr) = (s-1)\sum_{i=1}^k \hgt(y_i) - k(s-2),\]
which, using that $k \leq s$, yields $q(x) \le (s-1)(N+k-2)-k(s-2) \le (s-1)N-(s-2)$, as wished. 
\end{proof}

While the artinianity of a premonoid is a sufficient condition for its factorability, the same condition is far from being necessary. For example, if $H$ is the multiplicative monoid of the non-zero elements of the integral domain constructed by A.~Grams in \cite[Sect.~1]{Gra74}, then the divisibility premonoid of $H$ is factorable despite it is not Artinian (see \cite[Sect.~3]{Tr20(c)} for further details). However, the following result shows that Theorem~\ref{thm: FTF} is, in a sense, ``best possible''.

\begin{proposition}[{\cite[Proposition 2.8]{Cos-Tri-2023(a)}}] Let $A$ and $S$ be subsets of a monoid $H$ such that $1_H\notin A\cup S$. Then the following conditions are equivalent:
\begin{enumerate}[label=\textup{(\alph{*})}, mode=unboxed]
    \item Every element of $S$ factors as a non-empty product of elements of $A$.
    
\vskip 0.05cm

    \item There exists an Artinian preorder $\preceq$ on $H$ such that every element of $S$ is a $\preceq$-non-unit of $H$ and every $\preceq$-irreducible of $H$ is an element of $A$.
\end{enumerate}
\end{proposition}

Thus, in a monoid $H$, proving that every element of a given subset $S\subseteq H$ factors through the elements of another subset $A\subseteq H$ is equivalent to the artinianity of a suitable premonoid $(H,\preceq)$.

 In \cite[Corollary 2.5]{Tri23(a)}, Tringali gave a characterization of factorable premonoids that is based on a ``local'' version of Theorem~\ref{thm: FTF}. As a consequence \cite[Corollary 2.6]{Tri23(a)}, every non-unit of an acyclic monoid $H$ is a product of atoms (i.e., $H$ is atomic) if and only if $H$ admits a set of generators satisfying the ACCP. Recall that an element $x$ in a monoid $H$ satisfies the ACCP if there is no sequence $(x_k)_{k\ge 0}$ in $H$ with $x_0 = x$ and $Hx_kH\subsetneq Hx_{k+1}H$ for each $k\in \mathbb N$.

\section{A review of the known applications}\label{sec: examples}

In this section, we review a series of applications of the abstract factorization theorems \ref{thm: FTF} and \ref{thm: strongly Artinian}, all of which, except that in subsection \ref{Lasker-Noether}, have already been discussed in previous papers \cite{Tr20(c)} and \cite{Co-Tr-21(a)}. We direct the reader to these works for more details.

\subsection{Classical theorems}\label{sub: classical} A well-known result, stated in \cite[Proposition 0.9.3]{Cohn06}, establishes that a cancellative monoid satisfying ACCPR and ACCPL is atomic. By applying Theorem~\ref{thm: FTF} to a special class of divisibility premonoids, one can extend this classical factorization theorem to the unit-cancellative case:

\begin{corollary}
    Let $H$ be a unit-cancellative monoid satisfying the ACCPL and the ACCPR. Then every non-unit of $H$ is a product of atoms.
\end{corollary}
\begin{proof} By \cite[Corollary 4.6]{Tr20(c)}, $H$ satisfies the ACCP and, by Theorem~\ref{thm: FTF} (applied to $H_{\rm div}$) and Remark~\ref{rem:premonoids}\ref{rem:premonoids(1)}, $H$ is atomic (see also \cite[Theorem 2.28(i)]{Fa-Tr18}). 
\end{proof}

Anderson and Valdes-Leon established in \cite[Theorem 3.2]{And-ValLeo-1996} a generalization of the aforementioned results in a different direction. Working in the context of commutative rings with zero-divisors, they essentially showed that every non-unit of a {\it commutative} monoid satisfying the ACCP, even one that is possibly ``highly non-cancellative'', can be factored into a finite product of irreducibles. This result can be immediately generalized to the non-commutative case by applying Theorem~\ref{thm: FTF} to the divisibility premonoid of a Dedekind-finite monoid.

\begin{corollary}[{\cite[Corollary 4.1]{Tr20(c)}}]
    Let $H$ be a Dedekind-finite monoid satisfying the ACCP. Then every non-unit of $H$ is a product of irreducibles.
\end{corollary}

\subsection{Direct-sum decompositions of modules} In \cite[Corollary 4.13]{Tr20(c)}, Tringali proves out of Theorem~\ref{thm: FTF} an ``object decomposition result'' for special categories admitting finite products. This yields as a special case the following Corollary~\ref{corollary cat}, for which we present an alternative direct proof. We first recall some basic definitions from \cite[Chapter 6]{Lam99}. Given a ring $R$, we let the \evid{uniform dimension} ${\rm dim}_R(M)$ of a (left) $R$-module $M$ be the supremum of the positive integers $n$ for which there exist $n$ non-zero $R$ submodules $N_1, \ldots, N_n$ of $M$, such that $N_1\oplus_R \cdots\oplus_R N_n$ embeds into $M$. Here $\oplus_R$ denotes the direct sum of $R$-modules and we take $\sup\emptyset:=0$. It is well known (see, e.g., \cite[Corollary 6.10]{Lam01}) that the uniform dimension is {\it additive}, meaning that
\begin{equation}\label{eq: uniform dimension}
    {\rm dim}_R(M\oplus_R N)={\rm dim}_R(M)+{\rm dim}_R(N)\text{ for all }R\text{-modules }M, N.
\end{equation}

\begin{corollary}[{\cite[Corollary 4.14]{Tr20(c)}}]\label{corollary cat}
   Every $R$-module of finite uniform dimension over a ring $R$ is equal to a direct sum of finitely many indecomposable $R$-modules.
\end{corollary}
\begin{proof}
Given a ring $R$, let $H$ be the quotient of the set of $R$-modules with finite uniform dimension by the equivalence relation that identifies two isomorphic modules. We denote by $[M]$ the equivalence class of an $R$-module $M$, and observe that if $M$ and $N$ are two isomorphic modules, then ${\rm dim}_R(M)={\rm dim}_R(N)$. The set $H$ is then a monoid with respect to the operation $[M]\oplus[N]:=[M\oplus_R N]$, whose identity is the class $[0]$ of the zero $R$-module. For every $[M],[N]\in H$, set 
\[[M]\preceq [N]\text{ if and only if }{\rm dim}_R(M)\le{\rm dim}_R(N).\]
Clearly, $\preceq$ is an Artinian preorder on $H$, $[0]$ is the unique $\preceq$-unit of $H$, and, by \eqref{eq: uniform dimension}, the $\preceq$-irreducibles of $H$ (that are also $\preceq$-quarks) are exactly the classes of indecomposable modules, i.e., non-zero $R$-modules that are not the direct sum of two non-zero $R$-modules. Thus, we conclude from Theorem~\ref{thm: FTF} applied to the premonoid $(H,\preceq)$, that every $R$-module is isomorphic and hence equal to a direct sum of finitely many indecomposable $R$-modules. 
\end{proof}

Note that Corollary~\ref{corollary cat} applies, in particular, to every Artinian or Noetherian $R$-module, for which the direct sum decomposition into indecomposable $R$-modules is well known.

\subsection{The Lasker-Noether theorem}\label{Lasker-Noether} Let $R$ be a commutative ring and let $H:=\mathcal{I}(R)$ be the set of all ideals of $R$. Clearly, $H$ is a monoid with respect to the restriction $\cap_H$ to $H$ of the intersection on the power set $\mathcal{P}(R)$ of $R$. For for two ideals $I,J$ of $R$, define
\[I\preceq J \,\text{ if and only if }\, I\cap_H X=J \text{ for some }X\in H \,\text{ if and only if } \,J\subseteq I.\]
The relation $\preceq$ is clearly an order (i.e., a preorder with the antisymmetric property) on $H$, and we can consider the premonoid $\mathcal{H}=(H,\preceq)$. The unique unit of $\mathcal{H}$ is exactly $1_H=R$, a proper ideal $I$ of $R$ is a quark if and only if it is a maximal ideal, and $I$ is an irreducible if and only if it cannot be written as an intersection of two ideals properly containing $I$ (i.e., if and only if $I$ is an {\it irreducible ideal}). It is now easy to see the Lasker-Noether theorem, recalled in the introduction, as a corollary of Theorem~\ref{thm: FTF}.

\begin{corollary}[Lasker-Noether theorem]
    In a Noetherian commutative ring, every proper ideal can be decomposed as an intersection of finitely many primary ideals.
\end{corollary}
\begin{proof}
    In the notation above, the premonoid $\mathcal{H}$ is Artinian if and only if $R$ satisfies the ascending chain condition on its ideals, namely, if and only if $R$ is Noetherian. The claim follows then from Theorem~\ref{thm: FTF}, once recalled that every proper irreducible ideal of a Noetherian commutative ring is primary. 
\end{proof}

\medskip

Analogously, let $r$ be a weak ideal system on a commutative monoid $M$, as defined in \cite[Chapter 2]{Hal-Ko98}. The set $H_r:=\mathcal{I}_r(M)$ of $r$-ideals of $M$ is a monoid with respect to the restriction $\cap_{H_r}$ to $H_r$ of the intersection on $\mathcal{P}(M)$ (see \cite[(viii) on p.~16]{Hal-Ko98}). Moreover, the restriction $\subseteq_{H_r}$ to $H_r$ of the inclusion on $\mathcal{P}(M)$ is an order on $H_r$. Therefore, the pair $\mathcal{H}_r:=(H_r, \subseteq_{H_r})$ is a premonoid, and, whenever it is Artinian, every proper $r$-ideal of $M$ is an intersection of finitely many {\it irreducible} $r$-ideals, i.e. $r$-ideals that are not intersection of strictly bigger $r$-ideals. Note, in fact, that $M$ is the unique unit of $\mathcal{H}_r$. Moreover, a proper $r$-ideal of $M$ is a quark of $\mathcal{H}_r$ if and only if it is a maximal $r$-ideal. It is easily seen that when $M$ is $r$-Noetherian, then the premonoid $\mathcal{H}_r$ is Artinian. Thus, we can recover from Theorem~\ref{thm: FTF} the following:

\begin{corollary}[{\cite[Proposition on page 74]{Hal-Ko98}}] 
   If $M$ is an $r$-Noetherian monoid, then every proper $r$-ideal of $M$ can be written as the intersection of a finite number of irreducible $r$-ideals.
\end{corollary}

\subsection{Power monoids} Other applications of Theorems \ref{thm: FTF} and \ref{thm: strongly Artinian} include an improvement \cite[Proposition 4.11 and Theorem 4.12]{Tr20(c)} of a characterization due to Antoniou and Tringali \cite[Theorem 3.9]{An-Tr18} of the atomicity of reduced power monoids, a special class of monoids of sets that is ``highly non-cancellative''. In more detail, let $M$ be a multiplicative monoid. Equipped with the operation of setwise multiplication
$$
(X, Y) \mapsto XY := \{xy : x \in X, \, y \in Y\},
$$
the family of all non-empty finite subsets of $M$ containing the identity $1_M$ is itself a monoid, herein denoted by $\mathcal P_{\fin,1}(M)$ and called the \evid{reduced} \evid{power monoid} of $M$. The next proposition is a simplified version of \cite[Proposition 4.11]{Tr20(c)}.

\begin{proposition}
    Let $M$ be a monoid and let $H$ be the reduced power monoid $\mathcal P_{\fin,1}(M)$ of $M$. Then, the following hold:
    
\vspace{0.05cm}

    \begin{enumerate*}[label=\textup{(\roman{*})}, mode=unboxed]
        \item\label{i} $H$ is a Dedekind-finite monoid and $H^\times=\{1_M\}$.
    \end{enumerate*}

\vspace{0.05cm}

     \begin{enumerate*}[label=\textup{(\roman{*})}, mode=unboxed, resume]
        \item\label{ii} The divisibility premonoid $H_{\rm div}$ is Artinian and every $X\in H$ factors  as a product of irreducibles.
    \end{enumerate*}
\end{proposition}
\begin{proof}
    If $X$ divides $Y$ in $H$, then $X \subseteq Y$. In fact, if $Y=UXV$ for some $U,V\in H$, then $X=\{1_M\}X\{1_M\}\subseteq UXV=Y$. Therefore, $XY = \{1_M\}$ implies that $X \cup Y \subseteq \{1_M\}$ and hence $X = Y = \{1_M\}$. This proves item \ref{i}. Moreover, a sequence $(X_k)_{k\ge 0}$ of elements of $H$ that is non-increasing with respect to the divisibility preorder, is also non-increasing with respect to inclusion. Since the elements of $H$ are finite sets, this can only happen if $X_{k+1}=X_k$ for all large $k\in \mathbb{N}$. It turns out that $H_{\rm div}$ is Artinian, so the rest of the claim follows from Theorem~\ref{thm: FTF} and Remark~\ref{rem:premonoids}\ref{rem:premonoids(1)}, which ensures that $\mid_H$-irreducibles of the Dedekind-finite monoid $H$ are ordinary irreducibles. 
\end{proof}

\subsection{Idempotent factorization of matrices} In the paper \cite{Co-Tr-21(a)}, Cossu and Tringali analyze a preorder quite different from those in the previous examples, referred to as the $\rfix$-preorder. They consider this preorder within the multiplicative monoid of broad classes of right Rickart rings, i.e., rings in which the right annihilator of any element is generated by an idempotent. In particular, as a consequence of this analysis, they obtain from Theorem~\ref{thm: strongly Artinian} quantitative refinements \cite[Proposition 5.9 and Theorem 5.19]{Co-Tr-21(a)} of some classical results by Erdos \cite{Er68}, Dawlings \cite{Da81}, Laffey \cite[Theorem 1]{Laf83}, and Fountain \cite[Theorem 4.6]{Fo91} on {\it idempotent factorization} in the multiplicative monoid of the ring of $n$-by-$n$ matrices over a skew field or a commutative DVD. For more details on the $\rfix$-preorder and further applications of the factorization theorems of premonoids to factorizations into idempotents (and thus ``highly non-classical factorizations''), we refer the reader to the next section.

\section{Idempotent factorizations in \textup{r.\textsc{f}ix}-premonoids}\label{sec: idempotent}

As recalled in the previous section, applications of Theorem~\ref{thm: FTF} and Theorem~\ref{thm: strongly Artinian} include existence results for factorizations into atoms and irreducibles (thus related to the context of classical factorization), but also for factorizations into idempotent elements of non-commutative and non-cancellative monoids. This attests to how the language of premonoids allows the study of factorization problems of very different nature within the same theoretical framework. In this section of the paper, we delve into further ``non-classical'' applications of the aforementioned theorems, all leaning on a specific preorder. We recover the terminology from \cite[Section 3]{Co-Tr-21(a)}.

\begin{definition}\label{def:rfix}
Given a monoid $H$ and an element $a \in H$, we set 
\[\rfix_H(a) := \{x \in H : ax = x\}\] 
and let the \evid{$\rfix$-preorder on $H$} be the binary relation $\preceq$ on $H$ defined, for every $b,c\in H$, by 
\[b \preceq c \text{ if and only if }\rfix_H(c) \subseteq \rfix_H(b).\]
If $\preceq$ is the $\rfix$-preorder on $H$, we will refer to the premonoid $H_{\rfix}:=(H,\preceq)$ as the \evid{$\rfix$-premonoid} of $H$. Given a subpremonoid $\mathcal{K}=(K,\preceq_K)$ of $H_{\rfix}$, we will say that an element of $K$ is an $\rfix_H$\evid{[-non]-unit}, an \evid{$\rfix_H$-quark}, or an \evid{$\rfix_H$-irreducible [of degree $s$]} of $K$ if it is a [non]-unit, a quark, or an irreducible [of degree $s$] of $\mathcal K$, respectively. Clearly, the unique unit of $H_{\rfix}$ is the identity $1_H$.
\end{definition}

Cossu and Tringali investigated in \cite{Co-Tr-21(a)}, $\rfix$-premonoids of multiplicative monoids of special rings, including matrix rings. In particular, they considered factorizations in subpremonoids admitting idempotent irreducibles. In this section we focus on the $\rfix$-premonoid of a monoid $H$ that is {\it not} the multiplicative monoid of a ring, namely, the full transformation monoid of a finite set $X$. We will prove, in particular, that the irreducibles in the $\rfix_H$-subpremonoid of non-invertible maps from $X$ to $X$ are idempotent, while they are transpositions in the $\rfix_H$-subpremonoid of permutations of $X$.

\subsection{The \textup{r.\textsc{f}ix}-premonoid of the full transformation monoid}\label{subsect:howie}

Let $X$ be a finite set with cardinality $|X|=n\geq 1$, say $X:=\{1,\ldots,n\}$. From now till the end of the section, $H$ will denote the \evid{full transformation monoid of $X$}, i.e., the set of all the functions $h:X\to X$, $x\mapsto h(x)$, with the operation of composition. For $f,g\in H$, we write $fg$ to denote the composition $f\circ g$, and $f=g$ if $f(x)=g(x)$ for every $x\in X$. Moreover, for a transformation $h\in H$, we let $\fix(h)$ be the set of the elements $x\in X$ such that $h(x)=x$, and $h(X)$ be the image of $h$. Thus, it is readily seen that
\begin{equation}\label{eq:1}
    h\in\rfix_H(f)=\{g \in H : fg=g\} \text{ if and only if }h(X)\subseteq \fix(f).
\end{equation}

The following lemma gives a characterization of the $\rfix$-preorder on $H$, useful for our context.
\begin{lemma}\label{fix-preorder}
Let $\preceq$ be the $\rfix$-preorder on $H$. Then, for every $f,g\in H$, \[f\preceq g \text{ if and only if }\fix(g)\subseteq\fix(f).\]
\end{lemma}
\begin{proof}
By the definition of the $\rfix$-preorder, for $f,g\in H$, $f\preceq g$ if and only if $\rfix_H(g)\subseteq\rfix_H(f)$. 
Let us assume that $\rfix(g)\subseteq\rfix(f)$ and prove that $\fix(g)\subseteq \fix(f)$. If $\fix(g)=\emptyset$ the conclusion is trivial. If $\fix(g)\ne\emptyset$, up to a reordering of the elements of $X$, we can assume $\fix(g)=\{1,\dots,d\}$ with $1\le d\le n$. Define $h\in H$ by $h(i)=i$ for $i\in\llb 1,d\rrb$, $h(i)=1$ for $i\in \llb d+1,n\rrb$. Then $gh=h$, i.e., $h\in\rfix_H(g)$ and, by hypothesis, $h\in\rfix_H(f)$. It then follows by \eqref{eq:1} that $h(X)=\fix(g)\subseteq \fix(f)$. 
On the other hand, assume $\fix(g)\subseteq \fix(f)$ and let $h\in \rfix(g)$. Then $h(X)\subseteq\fix(g)\subseteq\fix(f)$, and hence $h\in \rfix_H(f)$ again by \eqref{eq:1}. 
\end{proof}

\begin{remark}\label{rem: rfix-height} 
Since the full transformation monoid $H$ of a finite set $X$ is a finite monoid, the $\rfix$-premonoid $H_{\rfix}$ is strongly Artinian in light of Remark~\ref{rem:height}\ref{rem:height(0)}. Given a non-identity map $h\in H$, if $h_1, h_2, \dots, h_\ell$ are non-identity maps of $H$ (so, non-units of $H_{\rfix}$) such that $h=h_1$ and $\fix(h_{k})\subsetneq \fix(h_{k+1})$ (i.e., $h_{k+1}\prec h_k$ with respect to the $\rfix$-preorder $\preceq$) for every $k\in\llb 1,\ell-1\rrb$, then $|\fix(h)|\le n+\ell$. It follows that $\hgt(h)\le n-|\fix(h)|$, where $n=|X|$. 
\end{remark}

\medskip

Following the same approach of \cite[Sections 4 and 5]{Co-Tr-21(a)}, we focus our attention to the submonoid $H^\sharp$ of $H$ consisting of the non-invertible (or singular) maps on $X$ and the identity ${\rm id}_X$, and we consider the subpremonoid of $H_{\rfix}$ with $H^\sharp$ as ground monoid. We call such premonoid the \evid{singular subpremonoid of $H_{\rfix}$}. Note that if $|X|:=n=1$, then $H=H^\sharp=\{{\rm id}_X\}$, so we assume from now on that $n\ge 2$. 
As pointed out in Remark~\ref{rem:height}\ref{rem:height(1)}, every subpremonoid of a strongly Artinian premonoid is strongly Artinian. In particular, it is easy to verify that for every non-identity $h\in H^\sharp$, it is always possible to realize a sequence of $n-|\fix(h)|$ elements of $H^\sharp$ starting from $h$ and strictly decreasing with respect to the $\rfix$-preorder. Therefore, we get from Remark~\ref{rem: rfix-height} that, for every $h\in H^\sharp\setminus\{{\rm id}_X\}$,
\begin{equation}\label{eq: height}
   \hgt(h)=n-|\fix(h)|, 
\end{equation}
where by $\hgt(\cdot)$ we mean the height relative to the singular subpremonoid of $H_{\rfix}$. We then gather from Remarks \ref{rem:height}\ref{rem:height(1)} and \ref{rem:height}\ref{rem:height(2)} that the singular subpremonoid of the $\rfix$-premonoid of $H$ is strongly Artinian and hence Artinian. Thus, Theorem~\ref{thm: FTF} ensures that each non-identity element of $H^\sharp$ factors into a functional composition of $\rfix_H$-irreducibles. We will give a characterization of $\rfix_H$-irreducibles of $H^\sharp$ in Proposition~\ref{pro: idempotent irreducibles} but first, we characterize its $\rfix_H$-quarks. For, a preliminary definition is in order: a transformation $\tau \in H$ is  a \evid{quasi-identity} if $|\fix(\tau)|=n-1$. Every quasi-identity is obviously i\-dem\-po\-tent and non-invertible.

\begin{proposition}\label{prop: quarks}
 A singular transformation $\alpha\in H^\sharp$ is an $\rfix_H$-quark if and only if it is a quasi-identity.   
\end{proposition} 
\begin{proof}
   By Remark~\ref{rem:height}\ref{rem:height(1)}, $\alpha$ is an $\rfix_H$-quark if and only if $\hgt(\alpha)=1$. By \eqref{eq: height}, this happens if and only if $|\fix(\alpha)|=n-1$, i.e., $\alpha$ is a quasi-identity. 
\end{proof}

\begin{proposition}\label{pro: idempotent irreducibles}
	Every $\rfix_H$-irreducible of $H^\sharp$ is a quasi-identity.
\end{proposition}
\begin{proof}
	Let $\alpha\in H^\sharp$ be an $\rfix_H$-irreducible and assume as a contradiction that $|\fix(\alpha)|<n-1$. If $n=2$ the result is trivial, so assume $n\ge 3$.
	Set $d := |\fix(\alpha)|$, $r := |\alpha(X)|$, and $\alpha(i):=a_i$ for every $i \in \llb 1, n \rrb$. Up to a suitable enumeration of the elements of $X$ we can assume that $a_i=i$ for $i\in\llb 1,d\rrb$, $a_i\ne i$ for $i\in \llb d+1,n\rrb$ and $a_i\in\llb 1,r\rrb$ for all $i$ (i.e., $\alpha(X)=\{1,\ldots,r\}$). Under our assumptions we have $0\leq d\leq \min\{r,n-2\}$, $1\leq r\leq n-1$ and $\llb d+1,n-1\rrb\ne \emptyset$. We distinguish two cases.
 
\medskip

\begin{enumerate*}[label=\textsc{Case} \arabic{*},mode=unboxed]
\item\label{Case1}: $\alpha^2(X)\subsetneq\alpha(X)$ or $r<n-1$.

    In this case there exists $j\in \llb d+1, n-1\rrb$ such that $a_j=a_i$ for some $i\in \llb 1,r\rrb$ strictly smaller than $j$. Let us then define $\beta,\gamma\colon X\to X$ by:
	\begin{equation}\label{eq: case 1}  
	\beta(i) := \left\{
	\begin{array}{ll}
	i & \text{if } i\in \llb 1,n-1 \rrb, \\
	a_n & \text{if } i = n,
	\end{array}
	\right.\quad\text{and}\quad
	\gamma(i) :=
	\left\{
	\begin{array}{ll}
	i & \text{if } i\in \llb 1,d \rrb\cup\{n\}, \\
	a_i & \text{if } i\in \llb d+1,n-1 \rrb.
	\end{array}
	\right.
	\end{equation}
	The maps $\beta$ and $\gamma$ are both singular: $\beta$ is a quasi-identity, and $\gamma(i)=\gamma(j)$. Moreover $\fix(\alpha)\subsetneq \fix(\beta), \fix(\gamma)$. It is readily verified that $\alpha=\beta\gamma$, and this contradicts the $\rfix_H$-irreducibility of $\alpha$.
\end{enumerate*}

\medskip

\begin{enumerate*}[label=\textsc{Case} \arabic{*},mode=unboxed,resume]
\item\label{Case2}: $\alpha^2(X)=\alpha(X)$ and $r=n-1$.
	
	Note that if $d+1=r$, then $\alpha^2(X)=\alpha(\{1,\ldots,d,r\})=\{1,\ldots,d,a_r\}=\{1,\ldots,d,r\}=\alpha(X)$ implies that $a_r =r$, which is impossible. Thus, it must be $d+2\le r$. Moreover, since $\alpha^2(X)=\alpha(X)$, for every $j\in\llb d+1,r\rrb$ there is a unique $k\in\llb d+1,r\rrb\setminus \{j\}$ such that $a_k=j$. We distinguish two subcases depending on whether or not $a_n$ is in $\llb 1,d\rrb$.
 \begin{enumerate}[label=\textup{(\alph{*})}, mode=unboxed]
 \item\label{2a} If $a_n\in\llb 1,d\rrb$, let $k\in\llb d+2,r\rrb$ be such that $a_k=d+1$. Define $\delta,\eta\colon X\to X$ by:
	\begin{equation}\label{eq: case 2a}  
 \delta(i) := \left\{
	\begin{array}{ll}
	i & \text{if } i\in \llb 1,r\rrb\setminus\{d+1\},\\
	a_i & \text{if } i=d+1,\\
	d+1 & \text{if } i = n,
	\end{array}
	\right.\quad\text{and}\quad
	\eta(i) :=
	\left\{
	\begin{array}{ll}
	i & \text{if } i\in \llb 1,d+1 \rrb,\\
	a_i & \text{if } i\in \llb d+2,r \rrb\setminus\{k\}, \\
	n & \text{if } i=k, \\
	a_n & \text{if } i=n.
	\end{array}
	\right.
   \end{equation}
	Both $\delta$ and $\eta$ are singular maps: since $a_{d+1}\in\llb 1,r\rrb\setminus\{d+1\}$, then $\delta(d+1)=a_{d+1}=\delta(a_{d+1})$, and $\eta(n)=\eta(a_n)$. Moreover, $\fix(\alpha)\subsetneq \fix(\delta),\fix(\eta)$ and it is routine to verify that $\alpha=\delta\eta$. Therefore, $\alpha$ cannot be an $\rfix_H$-irreducible of $H^\sharp$. 

\item\label{2b} If $a_n\in\llb d+1,r\rrb$, let $k\in \llb d+1,r\rrb\setminus\{a_n\}$ be such that $a_k=a_n$. Define 
	\begin{equation}\label{eq: case 2b}  
 \zeta(i):=\left\{
	\begin{array}{ll}
	i & \text{if } i\in \llb 1,d\rrb\cup\{n\},\\
	n & \text{if } i=k,\\
	a_i & \text{if } i\in\llb d+1,r\rrb\setminus\{k\}.
	\end{array}
	\right.
 \end{equation}
	It is then immediate to see that $\alpha=\beta\zeta$, where $\beta$ is the quasi-identity in \eqref{eq: case 1}.
 We reach a contradiction since $\zeta$ is a singular map ($\zeta(n)=\zeta(k)$) and $\fix(\alpha)\subsetneq \fix(\zeta)$.
 \end{enumerate}
\end{enumerate*}
All the possible cases have been considered, thus we conclude that $\alpha$ must be a quasi-identity. 
\end{proof}

In light of Theorem~\ref{thm: FTF}, we then obtain as a corollary of Proposition~\ref{pro: idempotent irreducibles} the classical theorem by Howie on the idempotent generation of $H^\sharp$, already mentioned in the introduction.
\begin{corollary}[{\cite[Theorem I]{Ho66}}]\label{cor: Howie}
Given a finite set $X$ with cardinality $n\ge 1$, every non-invertible map from $X$ to $X$ is a composition of quasi-identities.
\end{corollary}

The above result can be refined in light of Theorem~\ref{thm: strongly Artinian}, that provides a bound on the minimum number of idempotent factors.

\begin{theorem}\label{thm: a refinement of Howie}
Given a finite set $X$ with cardinality $n\ge 1$, every non-invertible map $\alpha\colon X\to X$ is a composition of at most $2(n-|\fix(\alpha)|)-1$ quasi-identities.
\end{theorem}
\begin{proof}
Let $H$ be the full transformation monoid of $X$. In light of \eqref{eq: height} and Proposition~\ref{prop: quarks}, the claim follows from Theorem~\ref{thm: strongly Artinian} with $s=3$, once proved that the singular subpremonoid of $H_{\rfix}$ satisfies its hypotheses. Therefore, we show that for every $\alpha\in H^\sharp$ with $|\fix(\alpha)|<n-1$ there exist non-identity maps $\beta_1,\dots, \beta_k\in H^\sharp$, with $k\in\llb 2,3\rrb$, such that $\fix(\alpha)\subsetneq \fix(\beta_i)$ for each $i\in\llb 1,k\rrb$, $\alpha=\beta_1\cdots \beta_k$, and $\hgt(\beta_1)+\cdots+\hgt(\beta_k)\le \hgt(\alpha)+1$. Fix a map $\alpha\in H^\sharp$ such that $d := |\fix(\alpha)|\le n-2$. We can assume without loss of generality that $\alpha(i):=a_i\in \llb 1,r\rrb$ for every $i\in \llb 1,n\rrb$, with $r\le n-1$, $a_i=i$ for every $i\in \llb 1,d\rrb$, and $a_i\ne i$ for every $i\in \llb d+1,n\rrb$. From now on, the proof mirrors that of Proposition~\ref{pro: idempotent irreducibles}. In \ref{Case1}, $\alpha=\beta\gamma$, for $\beta$ and $\gamma$ as in \eqref{eq: case 1}. Since $\hgt(\beta)=1$ and $\hgt(\gamma)=n-(d+1)=\hgt(\alpha)-1$, then $\hgt(\beta)+\hgt(\gamma)=\hgt(\alpha)$. \ref{Case2}\ref{2b} works analogously: $\alpha=\beta \zeta$, for $\zeta$ as in \eqref{eq: case 2b}. Since $\hgt(\zeta)=n-(d+1)=\hgt(\alpha)-1$, then $\hgt(\beta)+\hgt(\zeta)=\hgt(\alpha)$. \ref{Case2}\ref{2a} is slightly different and requires a deeper analysis. In this case $\alpha=\delta \eta$, where $\delta$ and $\eta$ are defined in \eqref{eq: case 2a}. For our convenience, we recall that  $\delta(i) :=i$ for $i=\llb 1,n-1\rrb\setminus\{d+1\}$, $\delta(d+1):=a_{d+1}$, and $\delta(n):=d+1$. It is immediate to verify that $\delta=\delta_1\delta_2$, where $\delta_1$ and $\delta_2$ are two quasi-identities defined by $\delta_1(n):=d+1$ and $\delta_2(d+1):=a_{d+1}$. It then follows that $\alpha=\delta_1\delta_2\eta$, with $\hgt(\delta_1)+\hgt(\delta_2)+\hgt(\eta)=\hgt(\alpha)+1$. In fact $\hgt(\delta_1)=\hgt(\delta_2)=1$ and $\hgt(\eta)=n-(d+1)=\hgt(\alpha)-1$. This concludes the proof. 
\end{proof}

\medskip

Summing up, as a corollary of Theorems \ref{thm: FTF} and \ref{thm: strongly Artinian}, applied to the singular subpremonoid of the $\rfix$-premonoid $H_{\rfix}$ of the full transformation monoid $H$ of a finite set $X$, we obtain a classical theorem of Howie on the idempotent factorization of singular maps of a finite set $X$, along with a refinement of it. It is then natural to ask what kind of information we can derive from the characterization of $\rfix_H$-quarks and $\rfix_H$-irreducibles in the subpremonoid of $H_{\rfix}$ having the group of units $H^\times$ of $H$ as the ground monoid.
Since $H^\times$ is exactly the symmetric group of $X$, and since every permutation of $X$ is a composition of transpositions, it is reasonable to expect that every $\rfix_H$-irreducible of $H^\times$ is a transposition. We will actually prove something more in Theorem~\ref{thm: transpositions}. In the following, we rephrase in our terminology the results in \cite[Example 3.13]{Tr20(c)}.

 If $f\colon X\to X$ is a non-identity permutation of the set $X:=\{1,\dots,n\}$, then $n\ge 2$ and $|\fix(f)|\leq n-2$. In particular, $|\fix(f)|=n-2$ if and only if $f$ is a transposition. In light of this fact, it is immediate to see that that, for every $f\in H^\times \setminus \{{\rm id}_X\}$,
 \begin{equation}\label{Eq: height-permutation}
      \hgt(f)= n-1-|\fix(f)|,
 \end{equation}
where by $\hgt(\cdot)$ we mean the height relative to the subpremonoid of $H_{\rfix}$ having $H^\times$ as a ground monoid. Then next proposition characterizes $\rfix_H$-quarks in such premonoid.

\begin{proposition}\label{prop: transpositions}
 A permutation $\alpha$ of $X$ is an $\rfix_H$-quark of $H^\times$ if and only if $\alpha$ is a transposition.   
\end{proposition} 
\begin{proof}
   By Remark~\ref{rem:height}\ref{rem:height(1)}, $\alpha$ is an $\rfix_H$-quark if and only if $\hgt(\alpha)=1$. In light of \eqref{Eq: height-permutation}, this happens if and only if $|\fix(\alpha)|=n-2$, i.e., $\alpha$ is a transposition. 
\end{proof}

As a consequence of the above characterization, we are able to recover the following well known result as a consequence of Theorem~\ref{thm: strongly Artinian}. We will make use of the fact that $f^2=\id_X$, for every transposition $f$.
\begin{theorem}\label{thm: transpositions}
Given a finite set $X$ with cardinality $n\ge 1$, every non-identity permutation $\alpha$ of $X$ is a composition of at most $n-|\fix(\alpha)|-1$ transpositions.
\end{theorem}
\begin{proof}
We show that if $\alpha:X\to X$ is a non-identity permutation of $X$ that it not a transposition, then $\alpha=\beta\gamma$, with $\beta$ and $\gamma$ non-identity permutations of $X$ such that $\fix(\alpha)\subsetneq \fix(\beta)$, $\fix(\alpha)\subsetneq \fix(\gamma)$, and $\hgt(\beta)+\hgt(\gamma)=\hgt(\alpha)$. The claim follows then from Theorem~\ref{thm: strongly Artinian} (with $s=2$) and \eqref{Eq: height-permutation}. Set $d:=\fix(\alpha)$. Then, without loss of generality, we can assume $\alpha(i)=i$ for $i\in \llb 1,d\rrb$ and $\alpha(i)\ne i$ for $i\in \llb d+1,n\rrb$, with $d\le n-3$. Fix an element $j\in\llb d+1,n\rrb $ and let $\beta$ be the transposition defined by $\beta(j):=\alpha(j)$ and $\beta(\alpha(j)):=j$ (note that $\alpha(j)\in \llb d+1,n\rrb\setminus\{j\}$). Then $\alpha=\beta(\beta\alpha)$. We then conclude by defining $\gamma:=\beta\alpha$, since $\gamma(j)=\beta\alpha(j)=\beta(\alpha(j))=j$. 
\end{proof}

Factorization into idempotent or invertible elements in non-commutative and non-cancellative monoids cannot be studied using classical factorization theory techniques. Corollary~\ref{cor: Howie}, Theorem~\ref{thm: a refinement of Howie}, and Theorem~\ref{thm: transpositions}, along with the examples mentioned in Section \ref{sec: examples}, further attest to how the language of premonoids serves as a useful and intriguing tool for extending the boundaries of factorization theory.

\section{Invertible linear maps over a finite dimensional vector space}\label{sec: Cartan Dieudonné}
In \cite[Section 5]{Co-Tr-21(a)}, Cossu and Tringali investigated the $\rfix$-premonoid of the multiplicative monoid of the endomorphism ring of a free module of finite rank over a principal ideal domain (PID). 

Let $D$ be a PID, $M$ a free $D$-module of finite rank $n\ge 1$, and let $R$ be the ring ${\rm End}(M)$ of all $D$-module endomorphisms of $M$ endowed with the operation of functional composition. We gather in the the following remark some facts from \cite[ Section 5.2]{Co-Tr-21(a)} concerning $R$.

\begin{remark}\label{rem:endo}

\begin{enumerate*}[label=\textup{(\arabic{*})}, mode=unboxed]
\item\label{rem:endo 1} For every $f\in R$, $\fix(f):=\{x\in M: f(x)=x\}$ is a direct summand of $M$. Moreover, if $\fix(f)\subsetneq \fix(g)$ for some $g\in R$, then $\fix(g)=\fix(f)\oplus N$ for some $\{0\}\ne N\subseteq M$. 
\end{enumerate*}

\begin{enumerate*}[label=\textup{(\arabic{*})}, mode=unboxed,resume]
\item\label{rem:endo 2} If $\preceq$ denotes the $\rfix$-preorder on $R$ and $f,g\in R$, then $f\preceq g$ if and only if $\fix(g)\subseteq\fix(f)$. The unique unit of the $\rfix$ premonoid $R_{\rfix}$ of $R$ is then the identity $\id_M$. 
\end{enumerate*}

\begin{enumerate*}[label=\textup{(\arabic{*})}, mode=unboxed, resume]
\item\label{rem:endo 3} The premonoid $R_{\rfix}$ is strongly Artinian. For every $f\in R$, the height $\hgt(f)$ of $f$ with respect to $R_{\rfix}$ is finite. Namely, $\hgt(f)\leq n-\rk(\fix(f))$, where $n:=\rk(M)$.
\end{enumerate*}
\end{remark}

The main results in \cite[Section 5]{Co-Tr-21(a)} concern the singular subpremonoid of $R_{\rfix}$, i.e., the subpremonoid having the monoid $R^\sharp:=\{f\in R: \rk(f(M)\le n-1\}\cup\{{\rm id}_M\}$ as ground monoid. In light of the results in the last part of Section \ref{sec: idempotent}, it makes sense to ask what information can we get from the characterization of $\rfix_R$-quarks and $\rfix_R$-irreducibles in the group of units $R^\times$ of the endomorphism ring $R$ of $M$. The next proposition gives a characterization of $\rfix_R$-quarks of $R^\times$. 

\begin{proposition}\label{pro: quarks in endo}
An endomorphism $f\in R^\times$ is an $\rfix_R$-quark of $R^\times$ if and only if $\rk(\fix(f))=n-1$.
\end{proposition}

\begin{proof}
Let $f\in R$ be an invertible endomorphism that is an $\rfix_R$-quark of $R^\times$, and assume as a contradiction that $d:=\rk(\fix(f))\le n-2$. By Remark~\ref{rem:endo}\ref{rem:endo 1}, there exists a basis $\{e_1,\dots, e_n\}$ of $M$ such that $\{e_1,\dots,e_d\}$ is a basis of $\fix(f)$ and define $g(e_i)=e_i$ for $i\in\llb 1,n-1\rrb$, and $g(e_n)=-e_n$. Then $g\in R^\times\setminus \{\id_M\}$ and $\fix(f)\subsetneq\fix(g)$, a contradiction. On the other hand, if $f$ is an invertible endomorphism of $M$ such that $\rk(\fix(f))=n-1$, then the only element $g\in R^\times$ such that $\fix(f)\subsetneq \fix(g)$ is the identity $\id_M$. 
\end{proof}

Unlike the case of $\rfix_R$-quarks, providing a meaningful characterization of the $\rfix_R$-irreducibles over the entire submonoid $R^\times$, in the most general setting where $D$ is any PID, is much more challenging. We leave this problem open for future work and focus on a very particular case, where we manage to obtain the Cartan-Dieudonné theorem as a corollary of Theorem~\ref{thm: strongly Artinian}. For, let $D=\R$, $R:=\End(\R^n)$, and focus our attention to the subgroup $\mathcal{O}(n)$ of orthogonal transformations of $\R^n$, i.e., the invertible linear maps of $\R^n$ preserving the inner product $\langle \cdot \rangle$: $\langle f(x),f(y)\rangle=\langle x,y\rangle$ for every $f\in \mathcal{O}(n)$ and $x,y\in\R^n$. We collect some easy observations in the following remark.

\begin{remark}\label{rem: orthogonal}
The elements of $\mathcal{O}(n)$ having $\fix$ of dimension $n-1$ are exactly the \evid{reflections}, i.e., orthogonal transformations of $\R^n$ fixing a space of dimension $n-1$. It is easy to verify (the argument is analogous to that in Proposition~\ref{pro: quarks in endo}) that a non-identity orthogonal transformation $f$ of $\R^n$ is an $\rfix_R$-quark of $\mathcal{O}(n)$ if and only if $f$ is a reflection.
\end{remark}

The following result shows that also every $\rfix_R$-irreducible of $\mathcal{O}(n)$ is a reflection. We will use in the proof that $r^2=\id_{\R^n}$ for every reflection $r\in\mathcal{O}(n)$. 
\begin{proposition}
Every  $\rfix_R$-irreducible of $\mathcal{O}(n)$ is a reflection.
\end{proposition}
\begin{proof}
Assume by way of contradiction that $f$ is an $\rfix_R$-irreducible of $\mathcal{O}(n)$ such that $d:=\rk(\fix(f))\leq n-2$. Since $\R^n=\fix(f)\oplus W$ with $W\ne\{0\}$, there exists a non-zero element $w\in W$ such that $f(w)\ne w$. Set $u=f(w)-w$. Since $f(W)\subseteq W$, $u\in W$ and there exists a basis $\{e_1,\dots, e_{n-1},u\}$ of $\R^n$ such that $\{e_1,\dots,e_d\}$ is a basis of $\fix(f)$ and $\{e_{d+1}, \dots, e_{n-1}, u\}$ is a basis of $W$. Note that, under our assumptions, $\llb d+1,n-1\rrb\ne\emptyset$. Let $g: \R^n\to \R^n$ be the reflection that fixes the orthogonal space of $u$, namely, $g(e_i)=e_i$ for $i\in\llb 1,n-1\rrb$ and $g(u)=-u$. Using that $w,f(w)\in W$, $f(w)-w=u$ and $\|f(w)\|^2=\|w\|^2$, it is easy to see that $w=\sum_{i=d+1}^{n-1}a_1 e_1-u/2$ and $f(w)=\sum_{i=d+1}^{n-1}a_1 e_1+u/2$ for some $a_i\in\R$. 
It follows then that $f=g(gf)$ with $g$ and $gf$ non-identities of $\mathcal{O}(n)$ such that $\fix(f)\subsetneq \fix(g),\fix(gf)$, and this contradicts the $\rfix_R$-irreducibility of $f$. 
\end{proof} 

The proof of the previous theorem actually shows that every orthogonal map $f\in\mathcal{O}(n)$ that is neither the identity nor an $\rfix_R$-quark factors as $f=gh$ with $g$ and $h$ non-identity maps in $\mathcal{O}(n)$ satisfying the hypotheses of Theorem~\ref{thm: strongly Artinian} (for $s=2$). Therefore, we can recover as a corollary of our abstract result the classical theorem from the 1930s due to Cartan and Dieudonné recalled in the introduction.
\begin{corollary}[Cartan-Dieudonné]\label{cor: CD}
Every element of the orthogonal group $\mathcal{O}(n)$ is a composition of at most $n$ reflections.
\end{corollary}

The results obtained for the orthogonal group $\mathcal{O}(n)$ of $\R^n$ can be extended verbatim to the orthogonal group of an $n$-dimensional, non-degenerate symmetric bilinear space over a field with characteristic different from $2$. The example illustrated above, is however sufficient for the scope of this note.

\section*{Acknowledgments}
The author acknowledges the support received from the European Union's Horizon 2020 program (Marie Sk\l{}odowska-Curie grant 101021791) and the Austrian Science Fund FWF (grant DOI 10.55776/ P\-AT\-975\-6623). Additionally, she is a member of the Gruppo Nazionale per le Strutture Algebriche, Geometriche e le loro Applicazioni (GNSAGA) at the Italian Mathematics Research Institute (INdAM).

The author is grateful to the referee for their meticulous reading and insightful feedback, which significantly improved the paper's presentation. Finally, special appreciation goes to Salvo (Tringali), not only for his valuable comments during the drafting of this manuscript but particularly for the enriching collaboration over the past few years. Engaging in mathematics is always enjoyable; doing it while having fun is even nicer.

\end{document}